\documentclass[12pt,oneside]{amsart} 
\usepackage{amssymb,amscd}
\usepackage{euscript}
 
 \usepackage{color} 

      \theoremstyle{plain}
      \newtheorem{theorem}{Theorem}[section]
      \newtheorem{lemma}[theorem]{Lemma}
      \newtheorem{corollary}[theorem]{Corollary}
      \newtheorem{proposition}[theorem]{Proposition}
      \newtheorem{remark}[theorem]{Remark}
      
      \newtheorem{definition}[theorem]{Definition}        
          
\numberwithin{equation}{section}

      \makeatletter
      \def\@setcopyright{}
      \def\serieslogo@{}
      \makeatother

\def\df{\mbox{Diff}}

\def\A{\EuScript{A}} 
\def\B{\EuScript{B}} 
 
\def\E{\mathcal{V}}
\def\V{\mathcal{V}}
\def\n{\mathcal N}
\def\M{\mathcal M}
\def\T{\mathcal{T}}

\def\R{\mathbb R}
\def\Z{\mathbb Z}
\def\N{\mathbb N}
\def\dist{\text{dist}}

\def\Id{\text{Id}}
\def\e{\epsilon}
\def\a{\alpha}

\def\la{\lambda}

\def\tw{\tilde W}
\def\tg{\tilde g}

\def\drm{\mbox{Diff}^{\,r}(\M)}
\def\dqm{\mbox{Diff}^{\,q}(\M)}
\def\dtwom{\mbox{Diff}^{\,2}(\M)}
\def\donem{\mbox{Diff}^{\,1}(\M)}
\def\dkm{\mbox{Diff}^{\,k+1}(\M)}
\def\df{\mbox{Diff}}

\def\QED{\hfill\hfill{\square}}

\begin{document}

\author{Victoria Sadovskaya$^{\ast}$}

\address{Department of Mathematics, The Pennsylvania State University, 
University Park, PA 16802, USA.}
\email{sadovskaya@psu.edu}

\title[Boundedness and invariant metrics for diffeomorphism cocycles]
{Boundedness and invariant metrics for diffeomorphism cocycles over hyperbolic systems} 

\thanks{$^{\ast}$ Supported in part by NSF grant DMS-1764216}
\thanks{{\it Mathematical subject classification:}\, 37D20, 54H15}
\thanks{{\it Keywords:}\, Cocycle, diffeomorphism group, periodic orbit, hyperbolic system, symbolic system}


\begin{abstract} 

Let $\A$ be a H\"older continuous cocycle over a hyperbolic dynamical 
system with values in the group of diffeomorphisms of a compact manifold $\M$.
We consider the periodic data of $\A$, i.e., the set of its return values along the 
periodic orbits in the base. We show that if the periodic data of $\A$ is bounded in 
$\dqm$, $q>1$, then the set of values of the cocycle is bounded in $\drm$ for each $r<q$.
Moreover, such a cocycle is isometric with respect to a H\"older continuous family of
Riemannian metrics on $\M$.
\end{abstract}

\maketitle 

\section{Introduction and statement of the results}

Group-valued cocycles over hyperbolic systems  have been extensively studied 
starting with the work of A. Liv\v{s}ic \cite{Liv1,Liv2}, who obtained definitive results
for commutative groups and made some progress for more general ones.
The subsequent research was focused on  non-abelian groups, starting with compact
groups and proceeding to more general ones, see \cite{KtN} 
for  a  survey of results prior to 2008.  The case of  $GL(d,\R)$ and Lie groups
 is now relatively well understood,
see for example \cite{PW,LW,K11,KS10,S15}. Cocycles with values in diffeomorphism groups give
another important  and dynamically natural class, which is harder to analyze.

In this paper we consider the group $\drm$ of $C^r$ diffeomorphisms of a compact 
connected manifold $\M$ and study cocycles  over a hyperbolic dynamical system
with values in $\drm$. By a hyperbolic system  we mean  either a transitive 
Anosov diffeomorphism of a compact connected manifold, or a topologically mixing  
diffeomorphism of a locally maximal 
hyperbolic set, or a mixing subshift of finite type, see Section \ref{hyperbolic}.
We focus on the problem of obtaining  information
 about the cocycle from its periodic data, that is, its return values along the periodic orbits in the base.
This is one of the central questions for cocycles over hyperbolic systems.
The basic problem in this area is to show that a cocycle with the identity  periodic data
is cohomologous to  the identity cocycle.
 This was done in \cite{NT95,LW} under
additional assumptions on growth of the cocycle and recently in \cite{KP,AKL,Gu}
just from the periodic data.  


In this paper we obtain results for much broader class of cocycles, which
have uniformly bounded periodic data. We show that all iterates of such 
a cocycle remain uniformly bounded and the cocycle is, in a sense, a cocycle 
of isometries.


\begin{definition} Let $f$ be a homeomorphism of a compact metric space $X$
and let $A$ be a function from $X$ to $\drm$. 
The {\em $\drm$-valued cocycle over $f$ generated by }$A$ 
is the map $\A:\,X \times \Z \,\to \drm$ defined  by $\,\A(x,0)=\Id\,$ and for $n\in \N$,
 $$
\A(x,n)=\A_x^n = A(f^{n-1} x)\circ \cdots \circ A(x) \quad\text{and}\quad\,
\A(x,-n)=\A_x^{-n}= (\A_{f^{-n} x}^n)^{-1} .
$$
\end{definition}

\noindent Clearly, $\A$ satisfies the {\em cocycle equation}\,
$\A^{n+k}_x= \A^n_{f^k x} \circ \A^k_x$.
\vskip.1cm

Cocycles can be considered in any regularity. We say that a cocycle $\A$  is 
bounded, continuous, or H\"older continuous if this property holds for its generator $A(x)=\A_x$.
We call a set $S$  in $\drm$ {\em bounded}\, if $\| g \|_{C^r}$ and $\| g^{-1} \|_{C^r}$
 are bounded uniformly in $g\in S$. Here  $\| . \|_{C^r}$ denotes the usual $C^r$ norm
 adapted to the manifold setting, see Section \ref{Crdist}, where a distance
 $d_{C^r}$ between diffeomorphisms is also defined.
A $\drm$-valued cocycle $\A$ is {\em $\beta$-H\"older}, $0<\beta \le 1$, if 
 there exists $c>0$ such that
$$
 d_{C^r} (\A_x, \A_y) \le  \, c\, d_X (x,y)^\beta \quad\text{for all  }x,y \in X. 
$$
H\"older continuity is the most natural setting for cocycles over hyperbolic systems,
especially since we include non-smooth systems in the base.

\vskip.1cm
For a cocycle $\A$,  we consider the periodic data set $\A_P$  and the set of all values $\A_X$,
$$
\A_P=\{ \A_p^k:\; p=f^kp,\; p\in X,\; k\in \N \} \quad\text{and}\quad
\A_X=\{ \A_x^n:\;  x\in X,\; n\in \Z \}.
$$

Our main results are the following two theorems. The first one gives boundedness
 of the cocycle and the second one yields an invariant metric.

\begin{theorem} \label{main 1} 
Let $(X,f)$ be a hyperbolic system, let $\M$ be a compact 
connected manifold, let $k\in \N$, and 
let $\A$ be a bounded $\dkm$-valued cocycle over $f$ which is H\"older continuous 
as a $\dqm$-valued cocycle with $q=k+\gamma$, $0<\gamma<1$.
\vskip.1cm
{\bf (i)} If the periodic data set $\A_P$ is  bounded in $\dqm$, then 
the value set $ \A_X$ is 

\hskip.65cm bounded in $\drm$ for any $r<q$.
\vskip.1cm
{\bf (ii)} If $\A_P$ is  bounded in $\donem$, then 
$ \A_X$ is also bounded in $\donem$.
\end{theorem}
Boundedness of cocycles  allows, in particular, to obtain higher regularity
of a continuous transfer map between them by the results in \cite{NT96,NT98}.

\begin{theorem} \label{main 2} 
Let $(X,f)$ and  $\M$ be as above, and let
 $\A$ be a  bounded $\dtwom$-valued cocycle over $f$ that is $\beta$-H\"older continuous 
as a $\dqm$-valued cocycle with $q=1+\gamma$, $0<\gamma<1$.
If the periodic data set $\A_P$ is  bounded in $\dqm$, then
there exists a family of Riemannian metrics $\{\tau_x: \, x\in X\}$ on $\M$
such that 
$$
\A_x : (\M,\tau_x)\to (\M,\tau_{fx}) \quad\text{is an isometry for each }x\in X.
$$
 Moreover,  for any $\a<\gamma$  each $\tau_x$ is 
$\a$-H\"older continuous on $\M$
and depends H\"older continuously on $x$ in $C^\a$ distance with exponent $\beta(\gamma-\a)$.
\end{theorem}

We note that uniform boundedness of the periodic data is crucial in our theorems.
Indeed, S. Hurtado constructed in \cite{H} an example of 
$\df^\infty(S^3\times S^1)$-valued cocycles over a Bernoulli shift 
for which every periodic value $\A_p^k$ is an isometry of some $C^\infty$ Riemannian metric 
on $\M$, but the cocycle has positive exponential growth rate of derivatives.
Also, our $C^{1+\text{H\"older}}$ regularity assumption for the cocycle is essential as 
it is known that an isometry of any $\a$-H\"older Riemannian metric is  $C^{1+\a}$ \cite{T}.

\section{preliminaries}

\subsection{Hyperbolic dynamical systems}\label{hyperbolic} See \cite{KH} for details.$\;$
\vskip.1cm
\noindent{\bf Transitive Anosov diffeomorphisms.} 
A diffeomorphism  $f$ of  compact connected  manifold $X$
 is called {\it Anosov}\, if there exist a splitting 
of the tangent bundle $TX$ into a direct sum of two $Df$-invariant 
continuous subbundles $E^s$ and $E^u$,  a Riemannian 
metric on $X$, and a number $\la$ such that 
\begin{equation}\label{Anosov def}
\|D_x f (v^s)\| < \la < 1 < \la^{-1} <\|D_x f (v^u)\|
\end{equation}
for any $x \in X$ and unit vectors  
$\,v^s\in E^s(x)$ and $\,v^u\in E^u(x)$.
The stable and unstable subbundles $E^s$ and $E^u$
are tangent to the stable and unstable foliations 
$W^s$ and $W^u$. 
The local stable manifold of $x$, $W^{s}_{loc}(x)$, is 
a small ball in $W^s(x)$ centered at $x$  so that 
$$
d_X (f^nx, f^ny)\le \la^n d_X(x,y)  \quad\text{for all }y\in W^s_{loc}(x)\text{ and }n\in \N.
$$
Local unstable manifolds 
are defined similarly. We assume that they are small enough so that  
$\,W^{s}_{loc}(x) \cap W^{u}_{loc}(z)\,$ consists of a single point
for any sufficiently close $x,z\in X$. 

A diffeomorphism is said to be {\it (topologically) transitive} if there is a point $x$ in $X$
with dense orbit. All known examples of Anosov diffeomorphisms have this property.
 
\vskip.2cm


\noindent{\bf Mixing diffeomorphisms of locally maximal hyperbolic sets.}
More generally, let $f$ be a diffeomorphism of a manifold $\n$.
A compact $f$-invariant  set $X \subset \n$ is
called {\em hyperbolic} if there  exist a continuous $Df$-invariant splitting 
$T_X \n = E^s\oplus E^u$, and a Riemannian metric on an open set
$U \supset X$ such that \eqref{Anosov def} holds for all $x \in X$.
Local stable and unstable manifolds are defined similarly for any $x \in X$
 and we denote  their intersections with $X$ by $W^{s}_{loc}(x)$ and  $W^{y}_{loc}(y)$.
The set $X$ is called {\em locally maximal} if 
$X= \bigcap_{n\in \Z} f^{-n }(U)$ for some open set $U\supset X$. 
This property ensures that $W^{s}_{loc}(x) \cap W^{y}_{loc}(y)$ exists in $X$.
The map $f|_X$ is called {\em topologically mixing}\,
 if for any two open non-empty subsets $U,V$ of $X$
 there is $N\in \N$ such that $\, f^n(U)\cap V\ne \emptyset\,$ for all $n\ge N$.
Topological mixing implies transitivity.
 

\vskip.2cm
\noindent{\bf Mixing subshifts of finite type.}
Let $M$ be $k \times k$ matrix with entries from $\{ 0,1 \} $ such that all 
entries of $M^N$ are positive for some $N$. Let
$$
X= \{ \,x=(x_n) _{n\in \Z}\, : \,\; 1\le x_n\le k \;\text{ and }\;
 M_{x_n,x_{n+1}}=1 \,\text{ for every } n\in \Z \,\}.
$$
\noindent The shift map $f:X\to X\,$ is defined by 
$(f(x))_n=x_{n+1}$.
The system $(X,f)$ is called a {\em  mixing  subshift of finite type}. 
We fix $\la \in (0,1)$ and consider the metric 
$$
d(x,y) = d_\la (x,y)=\la^{n(x,y)},
\;\text{ where }\;n(x,y)=\min\,\{ \,|i|\,: \; x_i \ne y_i  \}.
$$
The following sets play the role of the local stable and unstable 
manifolds of $x$
$$
W^s_{loc}(x)=\{\,y: \;\, x_i=y_i, \;\;i\ge 0\,\}, \quad 
W^u_{loc}(x)=\{\,y:\;\, x_i=y_i, \;\;i\le 0\,\}.
$$
Indeed, for all $x\in X$, $y\in W^s_{loc}(x)$, $y'\in W^u_{loc}(x)$, and $n\in \N$,
$$
d (f^n x, f^n y )= \la^n \, d  (x,y) \quad\text{and}\quad
d  (f^{-n}x, f^{-n}y )= \la^n\, d (x,y),
$$
and for any $x, z\in X$ with $d(x,z) < 1$ the intersection of $W^s_{loc}(x)$ and $W^u_{loc}(z)$
consists of a single point, $y=(y_n)$ such that $y_n=x_n$ for $n\ge 0$
and $y_n=z_n$ for $n\le 0$.


\subsection{Distances on the space of diffeomorphisms.} \label{Crdist}
\vskip.2cm


We fix a smooth background Riemannian metric and the corresponding distance 
$d_\M$ on $\M$. Recall that $\drm$, $r\ge 1$, is the set of $C^r$ diffeomorphisms of $\M$. 
The $C^r$ topology on $\drm$ can
be defined using coordinate patches and the $C^r$ norm for the Euclidean space. 

 For diffeomorphisms $g$ and $h$ of $\M$ we set
$$
d_{C^0}(g,h)=\max_{t\in \M} d_\M(g(t),h(t))+ \max_{t\in \M} d_\M(g^{-1}(t),h^{-1}(t)).
$$
For any $g\in \drm$,  $r\in \N$, we can define its $C^r$ size  as
$$
  |g|_{C^r}= \|g\|_{C^r}+ \|g^{-1}\|_{C^r} , \quad \text{where }\; 
  \|g\|_{C^r}= \max_{t\in \M} d_\M(g(t),t)+ \max_{1 \le i \le r}\max_{t\in \M} \|D_t^ig\|, 
$$
where $D_t^ig$ is the derivative of $g$ of order $i$ at $t$ and its norm is defined as
the norm of the corresponding multilinear form from $T_t\M$ to $T_{g(t)}\M$ 
with respect the Riemannian metric. For $r=k+\a$, $k\in \N$,  $0<\a<1$, the definition is similar with
$$
 \|g\|_{C^r}=\|g\|_{C^k} + 
 \sup \,\{\, \|D^k_t -D^k_{t'}\|\cdot d_\M(t,t')^{-\alpha} : \;t, t'\in \M, \; 0< d_\M(t,t') <\e_0 \},
$$
where $\e_0$ is chosen small compared to the injectivity radius of $\M$ so that
the tangent bundle is locally trivialized via parallel transport and the difference makes sense.
\vskip.05cm 
We call a set $S$ {\it bounded} in $\drm$ if $\,\{ |g|_{C^r}: g\in S\}$ is bounded.
\vskip.05cm 
A natural distance $d_{C^r}(g,h)$ 
on $\drm$,  $r\in \N$,
was defined in \cite{LW} as the infimum of the lengths of piecewise $C^1$ paths  
in $\drm$ connecting $h$ with $g$  and $h^{-1}$ with $g^{-1}$,
where the length of a path $p_s$ is 
$$
\ell_{C^r} (p_s) =\max_{0 \le i \le r}\max_{t\in \M}\int \|\tfrac d{ds} (D_t^i \, p_s) \|\,ds, 
$$
see \cite[Section 5]{LW} for details. For $r=k+\a$,  one also adds 
the corresponding H\"older term:
$$
\ell_{C^r} (p_s) =\ell_{C^k} (p_s)
+ \max_{t\in \M}\int |\tfrac d{ds}  \| (D^k p_s) \|_{\a,t}|\,ds, \quad\text{where}
$$
$$
 \|D^k g \|_{\a,t}= \sup \,\{\, \|D^k_{t'} -D^k_{t}\|\cdot d_\M(t',t)^{-\alpha} : 
 \,\;t'\in \M, \;\, 0< d_\M(t,t') <\e_0 \}.
$$
We will consider distances only between sufficiently $C^r$-close
diffeomorphisms, $r>1$. Then the distance $d_{C^r}(g,h)$ is Lipschitz equivalent 
to $ \|g-h\|_{C^r}+ \|g^{-1}-h^{-1}\|_{C^r}$, where  the difference is understood 
using local trivialization. More specifically,
there exist  constants $\kappa$ and $\delta_0>0$ depending only on $r$ and 
the Riemannian metric so that 
\begin{equation}\label{compare 1}
\kappa^{-1} d_{C^r}(g,h) \le  \|g-h\|_{C^r}+ \|g^{-1}-h^{-1}\|_{C^r} \le \kappa \, d_{C^r}(g,h),
\end{equation}
provided that either $d_{C^r}(g,h)< \delta_0 |g|_{C^r}^{-1}\,$ or 
$\,\|g-h\|_{C^r}+ \|g^{-1}-h^{-1}\|_{C^r}<\delta_0 |g|_{C^r}^{-1}$.
The first inequality in \eqref{compare 1} can be obtained using an interpolating path 
$$
p_s(t)=\exp_{g(t)}(s\cdot \exp ^{-1}_{g(t)}h(t) )
$$ 
between $C^r$-close diffeomorphisms. It is analogous to the ``straight line homotopy" 
$g+s(h-g)$ with $\frac{d}{ds}p_s=h-g.$ The $C^r$ closeness of $g$ and $h$ ensures 
that each $p_s$ is a diffeomorphism as its differential is invertible.
The second inequality follows from the mean value theorem.  
When the diffeomorphisms are not necessarily close, we have a one-sided estimate
\begin{equation}\label{compare 2}
 |g|_{C^r} \le  \exp(\kappa' \, d_{C^0}(g,h))\cdot (|h|_{C^r} + d_{C^r}(g,h)).
\end{equation}
where the constant $\kappa'$ depends only on the Riemannian metric. 
It is given by \cite[Lemma 5.1]{LW} for $r\in \N$ and is obtained similarly 
for $r=k+\a$.


\section{Proof of Theorem \ref{main 1}}

\subsection{Slow growth of the cocycle} \label{slow}
First we show that the cocycle has slow growth using an extension of the 
following recent result by Avila, Kocsard, and Liu.
\begin{theorem} \label{AKL} \cite[Theorem 2.5]{AKL} 
Let $f: X \to X$ be a hyperbolic homeomorphism and let $\A: X \to \df \,^{1+\gamma}(\M)$, $\gamma >0$,
be a H\"older continuous cocycle. 

If  $\A_P=\Id$, that is the cocycle has the identity periodic data, then
\begin{equation}\label{0exp}
\lim_{n\to \pm \infty } n^{-1} \log \|  D_t\A_x^n  \| =0 \quad \text{for all }\,(x,t)\text{ in } X\times \M.
\end{equation}
\end{theorem}

\begin{proposition} \label{0expProp}
The conclusion of Theorem \ref{AKL} holds under the weaker assumption on 
the periodic data, that the set $\A_P$ is bounded in $\df \,^{1}(\M)$.
\end{proposition}
\begin{proof}
This can be seen in the first three paragraphs of the {\it Proof of Theorem 1.1} in Section 5
of \cite{AKL}. Indeed, assuming \eqref{0exp} does not hold, they show existence of a periodic
point $p=f^np$ in $X$ and $t\in \M$  for which some singular value of the derivative
$D_t\A_p^n$ is arbitrarily large. This contradicts boundedness of $\A_P$ in $\df \,^{1}(\M)$.
\end{proof}


We consider the  vector bundle  $\E$  over $X\times \M$ with fiber
$\E_{(x,t)} =T_t\M$ and the linear cocycle 
$$
\B_{(x,t)}= D_t\A_x \;\text{ on $\E\;$ over the map }\;F(x,t)=(fx, \A_x (t)).
$$
The iterates of $\B$ are 
given by  $\B^n_{(x,t)}: T_{t\,}\M \to T_{\A^n_x(t)\,}\M$, where
$$
\B^n_{(x,t)}= D_t\A_x^n =
D_{\A^{n-1}_x(t)}\A_{f^{n-1}x} \circ ... \circ D_{\A_x(t)}\A_{fx} \circ D_t\A_x.
$$

Since the periodic data set $\A_P$ is bounded in $\df \,^{q}(\M)$ and hence in $\df \,^{1}(\M)$,
Proposition \ref{0expProp} applies and yields \eqref{0exp}, which implies 
that all Lyapunov exponents of $\B_{(x,t)}$  are zero with respect to any $F$-invariant measure on 
$X\times \M$.  It is well-known that the latter implies sub-exponential growth of the norm, 
see e.g. \cite{Schr}.
More precisely, for each $\e>0$ there exists $K_\e$ such that for all $x\in X$,
$$
\sup\,\{\|D_t \A_x^n \|:t\in \M \}\le K_\e e^{|n|\e} \quad\text{for all }n\in \Z,
$$
where the case of $n<0$ follows from the similar  estimates of the inverse of $\A$.
Since the cocycle $\A_x$ is  bounded in $C^{k+1}$,  it follows that higher derivatives
also grow sub-exponentially, see e.g.  \cite[Lemma 5.5]{LW}:
for each $\e>0$ there exists $K_\e$ such that for all $x\in X$ and $1\le i \le k+1$,
$$
\sup\,\{\|D_t^i \A_x^n \|:t\in \M \}\le K_\e e^{|n|\e} \quad\text{for all }n\in \Z.
$$
Therefore, since $q\le k+1$,  for each $\e>0$ there exists $K_\e$ such that  for all $x\in X$,
$$
|\A_x^n |_{C^q} \le K_\e e^{|n|\e} \quad\text{for all }n\in \Z.
$$

\subsection{Holonomies.}$\;$
In this section we establish existence and regularity of holonomies for the cocycle $\A$. 
This is a fundamental property of cocycles with sufficiently slow growth. For $\dqm$-valued 
cocycles with integer $q \ge 2$, the holonomies in $\df \, ^{q-1}(\M)$ were constructed in \cite{BK} using 
estimates from in \cite{LW}. We consider an arbitrary $q>1$ and obtain holonomies in $\drm$ for any $r<q$. 
This allows us to obtain our main results with any $q>1$ and have an arbitrarily small loss of regularity.
Using the results from  \cite{BK} in our arguments would result in the loss of several derivatives as in 
 \cite[Theorem 1.1]{BK}. 
 
 Our proof follows the usual approach of showing that $\{(\A^n_y)^{-1} \circ \A^n_x\}$ is a Cauchy sequence
 by estimating the distances between consecutive terms. The main difficulty for cocycles of 
 diffeomorphisms is bounding the distortion of distances produced by composition on the left and on the right.
 These estimates are given in Lemmas \ref{dlLO lemma} and \ref{distance} below. The main difference with
  \cite{BK} is that we work with non-integer $q$ and $r$ and obtain a distortion estimate \eqref{dist} of 
  H\"older rather than Lipschitz type.  This allows us to loose arbitrarily small part of $q$,
  at the expense of obtaining lower H\"older regularity of the resulting holonomies along $W^s$ in (H3).
 
\begin{proposition} \label{holonomies} 
Let $\A$ be a $\dqm$-valued cocycle over a hyperbolic dynamical system, 
where $q=k+\gamma$ with  $k\in \N$ and $0<\gamma<1$. 
Let  $r=k+\alpha$, where $0<\alpha<\gamma $.
Suppose that $\A$ is $\beta$-H\"older in $\dqm$,
and for each $\e>0$ there exists $K_\e$ such that
\begin{equation} \label{C^q bunching}
  |\A_x^n|_{C^q} \le K_\e e^{\e n} \quad\text{ for all } n\in \N.
\end{equation}

Then for any $x\in X$ and $y\in W^s(x)$,
the limit 
\begin{equation} \label{holonomy def}
H^s_{x,y}=H^{\A,s}_{x,y}=\lim_{n\to+\infty} (\A^n_y)^{-1} \circ \A^n_x
\end{equation}
exists in $\drm$ and satisfies 
\begin{itemize}
\item[(H1)]  $H^s_{x,x}=\Id\,$ and $\,H^s_{y,z} \circ H^s_{x,y}=H^s_{x,z}$,\,\,
which imply $(H^s_{x,y})^{-1}=H^s_{y,x};$
\vskip.1cm
\item[(H2)]  $H^s_{x,y}= (\A^n_y)^{-1}\circ H^s_{f^nx ,\,f^ny} \circ \A^n_x\;$ 
for all $n\in \N$;
\vskip.1cm
\item[(H3)] There exists a constant $c$ such that for all  $x\in X$   and $y\in W^{s}_{\text{loc}}(x),$
$$
d_{C^r}(H^s_{\,x,y},\Id) \leq c\,d_X (x,y)^{\beta \rho}, \quad\text{where }\rho=q-r,
$$
 moreover,
\begin{equation} \label{Anx}
 d_{C^r}  \left( (\A^n_y)^{-1} \circ \A^n_x, \, \Id \right) \le  c\,  d_X(x,y)^{\beta\rho} 
  \quad\text{for all }n\in \N.
\end{equation}
\end{itemize}

 \end{proposition}

The maps  $H^s_{x,y}$ are called the {\it stable holonomies of $\A$}.
It is convenient to view  $H^s_{x,y}$ as a $C^r$ diffeomorphism from $\M_x$ to $\M_y$,
i.e., from the fiber at $x$ to the fiber at $y$.
The unstable holonomies are defined similarly: for $y\in W^u(x)$,
$$
H^u_{x,y}=H^{\A,u}_{x,y}=\lim_{n\to-\infty} (\A^n_y)^{-1} \circ \A^n_x
$$

\begin{remark} 
For the conclusion to hold, we do not need \eqref{C^q bunching} to be satisfied for each $\e>0$,
just for a sufficiently small one so that $e^{2\e (1+ r)} \la^{\beta\rho}<1$, where 
$\la$ is as in \eqref{Anosov def}.
\end{remark}

First we give some estimates for distances and norms. 
Lemma \ref{dlLO lemma} follows directly from Propositions 5.5 in \cite{dlLO},
where the estimates are obtained for compositions of smooth maps between 
open sets in Euclidean spaces.

\begin{lemma}  \cite{dlLO} \label{dlLO lemma}
For any $r\ge 1$ there exist constants  $M_r$ and $M_r'$ such that for any $h,g \in C^r(\M)$,
$$
\begin{aligned}
  \| h\circ g\|_{C^r} &\le M_r' \left(1+ \| g\|_{C^1}^{r-1} \right) 
  \left( \|h\|_{C^1}\| g\|_{C^r} + \|h\|_{C^r}\| g\|_{C^1} \right) +  \|h\|_{C^0} \le \\
 & \le M_r \,  \|h\|_{C^r} (1+\| g\|_{C^r})^r .
 \end{aligned}
  $$
\end{lemma}

The next lemma relies on further results  of \cite{dlLO} which involve 
differences of functions. We apply these results to close diffeomorphisms 
using estimate \eqref{compare 1}, and so we need to ensure sufficient 
closeness of the diffeomorphisms $h_1$ and $h_2$, as well as that of the compositions
$g\circ h_1\circ \tg$ and  $g\circ h_2 \circ \tg$.

\begin{lemma}\label{distance} 
Let $q=k+\gamma$, $r=k+\a$, and $\rho=q-r$, where $k\in \N$ and $0\le\a<\gamma\le1.$
There exists a constant $M=M(r,q,\M,K)$ such that 
for any $g, \tg \in \dqm$ and $h_1, h_2\in \drm$  with $|h_1|_{C^r}, |h_2|_{C^r}\le K,$
\begin{equation} \label{dist}
\begin{aligned}
&d_{C^r} (g\circ h_1\circ \tg,\, g\circ h_2 \circ \tg) \le \\
&\le M \left( \| g\|_{C^q} (1+\| \tg\|_{C^r})^r + \| \tg^{-1}\|_{C^q} (1+ \| g^{-1}\|_{C^r})^r \right)
 \cdot d_{C^r}(h_1, h_2)^\rho
\end{aligned}
\end{equation} 
provided that $d_{C^r}(h_1, h_2)\le \delta_0 |h_1|_{C^r}^{-1}$ and  the right hand side 
of \eqref{dist} is less than 
\begin{equation} \label{delta 0}
\delta_0 \left( M_r^{2}\, (1+| h_1|_{C^r})^{r}(   \|g\|_{C^r}  \, (1+\| \tg\|_{C^r})^r +
 \|\tg^{-1}\|_{C^r}   \, (1+\| g^{-1}\|_{C^r})^r )\right)^{-1}, 
\end{equation} 
 where $ \delta_0$ is as in \eqref{compare 1}.  
\end{lemma}

\begin{proof} 
We use \eqref{compare 1} to estimate the distance between close diffeomorphisms 
together with the following estimate from \cite[Propositions 6.2(iii)]{dlLO}:
$$
  \| g\circ  h_1 - g\circ h_2\|_{C^r} \le 
  M' \,\|g\|_{C^q} \,\|  h_1 -  h_2\|_{C^r} ^\rho,
$$
where $M'=M'(r,q,\M,K)$. Using this and Lemma \ref{dlLO lemma} we obtain
$$
\begin{aligned}
 & \|g\circ h_1\circ \tg-g\circ h_2\circ \tg\|_{C^r} = \|(g\circ h_1- g\circ h_2)\circ \tg\|_{C^r}\le \\
 &\le M_r \|g\circ h_1- g\circ h_2\|_{C^r} \cdot (1+ \|\tg\|_{C^r})^r  \le \\
 &\le M_r M' \,\|g\|_{C^q} \,\|  h_1 -  h_2\|_{C^r} ^\rho \cdot (1+ \|\tg\|_{C^r})^r 
 =M'' \,\|g\|_{C^q}(1+ \|\tg\|_{C^r})^r \cdot \|  h_1 -  h_2\|_{C^r} ^\rho
 \end{aligned}
$$
It follows that 
$$
 \begin{aligned}
&\|g\circ h_1\circ \tg-g\circ h_2\circ \tg\|_{C^r}+
\|\tg^{-1}\circ h_1^{-1}\circ g^{-1}-\tg^{-1}\circ h_2^{-1}\circ g^{-1}\|_{C^r}\le\\
&\le  M'' \, \| g\|_{C^q} \,(1+\| \tg\|_{C^r})^r\cdot  \|h_1-h_2\|_{C^r}^\rho \,+\\
& \qquad\qquad \qquad \qquad\quad
 +M'' \, \| \tg^{-1}\|_{C^q} \,(1+\| g^{-1}\|_{C^r})^r\cdot  \|h_1^{-1}-h_2^{-1}\|_{C^r}^\rho \le\\
 &\le M'' \left( \| g\|_{C^q} \,(1+\| \tg\|_{C^r})^r + \| \tg^{-1}\|_{C^q} \,(1+\| g^{-1}\|_{C^r})^r \right)
 \cdot \kappa\, d_{C^r}(h_1, h_2)^\rho.
 \end{aligned} 
$$
By \eqref{compare 1}, this yields the same estimate for 
$ \kappa^{-1} d_{C^r} (g\circ h_1\circ \tg,\, g\circ h_2 \circ \tg)$ and gives \eqref{dist} with 
$M=M''\kappa^2$, provided that this estimate is at most 
$\delta_0 \cdot |g\circ h_1\circ \tg|_{C^r}^{-1}$.  The latter follows from the last assumption 
of the lemma and the following estimate for $|g\circ h_1\circ \tg|_{C^r}$.
 Applying Lemma \ref{dlLO lemma} twice, we get
$$
\begin{aligned}
& \|g\circ h_1\circ \tg\|_{C^r} = \|(g\circ h_1)\circ \tg\|_{C^r}\le
 M_r \,  \|g\circ h_1\|_{C^r} (1+\| \tg\|_{C^r})^r  \le \\
 & \le  M_r^2\,  \|g\|_{C^r} (1+\| h_1\|_{C^r})^r  \, (1+\| \tg\|_{C^r})^r . 
  \end{aligned}
$$
Similarly, 
$$ 
\|(g\circ h_1\circ \tg)^{-1}\|_{C^r} 
\le  M_r^2\,  \|\tg^{-1}\|_{C^r} (1+\| h_1^{-1}\|_{C^r})^r  \, (1+\| g^{-1}\|_{C^r})^r , 
$$
 and hence
$$
 |g\circ h_1\circ \tg|_{C^r}  \le M_r^2\,  (1+| h_1|_{C^r})^r
 \left( \|g\|_{C^r} (1+\| \tg\|_{C^r})^r  +  \|\tg^{-1}\|_{C^r}  (1+\| g^{-1}\|_{C^r})^r \right ).
 $$
\end{proof}

\noindent {\bf Proof of Proposition \ref{holonomies}}.
We  fix  $x\in X$ and construct $H^s_{x,y}$ for any $y\in W_{loc}^s(x)$ sufficiently close to $x$. 
Then the map $H^s$ can be 
extended to the whole leaf $W^s(x)$ by the invariance property (H2).
\vskip.1cm

For  $y \in W_{loc}^s(x)$, we have $d_X (f^nx, f^ny)\le d_X(x,y) \lambda^n$  for all $n\in \N$,
where  $0<\lambda<1$ is as in \eqref{Anosov def}.
We consider the sequence of diffeomorphisms 
$\{ (\A^n_y)^{-1}\circ  \A^n_x \}$ and show that it is Cauchy in $\drm$.

We apply  Lemma \ref{distance} with $h_1=\Id$ and $h_2=  (\A_{f^ny})^{-1} \circ \A_{f^nx}\,$
to estimate the distance between consecutive terms. We will verify 
assumption \eqref{delta 0} later.
 $$
 \begin{aligned}
 & d_{C^r} \left( (\A_y^n)^{-1}  \circ \A_x^n ,\, (\A_y^{n+1})^{-1} \circ \A_x^{n+1}\right)=\\
 &= d_{C^r}\left( (\A_y^n)^{-1}\circ\Id\circ \A_x^n ,\,  (\A_y^n)^{-1}
 \circ(\A_{f^ny})^{-1} \circ \A_{f^nx}\circ \A_x^n\right) \le\\
&\le M  \left( \| (\A_y^n)^{-1}\|_{C^q} (1+\| \A_x^n\|_{C^r})^r + 
\| (\A_x^n)^{-1}\|_{C^q} (1+\| \A_y^n\|_{C^r})^r \right) \times\\
&\hskip6cm \times  d_{C^r} \left( \Id, (\A_{f^ny})^{-1} \circ \A_{f^nx} \right)^\rho.
 \end{aligned}  
 $$
 Using condition \eqref{C^q bunching} we estimate 
$$
\| (\A_y^n)^{-1}\|_{C^q}  (1+\| \A_x^n\|_{C^r})^r\le 
K_\e e^{\e n} (1+K_\e e^{\e n})^r \le K_\e' \,e^{\e n (1+ r)}.
$$
 Thus we have 
\begin{equation}\label{Cauchy}
 \begin{aligned}
 &d_{C^r} \left( (\A_y^n)^{-1}  \circ \A_x^n ,\, (\A_y^{n+1})^{-1} \circ \A_x^{n+1}\right)
 \le  K_\e'' e^{\e n (1+ r)} \cdot \dist (f^n y, f^n x)^{\beta\rho} \le\\
&\le K_\e''  e^{\e n (1+ r)} \cdot  ( d_X(x,y) \lambda^n)^{\beta\rho} 
  =K_\e''\, d_X(x,y)^{\beta\rho} \cdot \theta^n, \;\text{ where }\theta=e^{\e (1+ r)} \la^{\beta\rho}.
 \end{aligned}
\end{equation}
To verify assumption \eqref{delta 0}, we similarly estimate 
 $$
 \| (\A_y^n)^{-1}\|_{C^r} (1+\| \A_x^n\|_{C^r})^{r} + 
\| (\A_x^n)^{-1}\|_{C^r} (1+\| \A_y^n\|_{C^r})^{r}  \le 2K_\e'\, e^{\e n (1+ r)}.
 $$
  We take $\e>0$ such that $\theta \,e^{\e (1+ r)}<1$, and in particular $\theta <1$.
 Since 
 $$
 K_\e''\,d_X(x,y)^{\beta\rho} \cdot \theta^n <  K_\e''\,d_X(x,y)^{\beta\rho} \cdot e^{-n\e (1+ r)},
 $$
  we conclude that  \eqref{delta 0} is satisfied
of all $n\in \N$ provided that $d_X(x,y)$ is small enough.
\vskip.1cm

 Therefore, by \eqref{Cauchy}, \,$\{ (\A^n_y)^{-1}\circ  \A^n_x \}$ is a 
 Cauchy sequence in $\drm$, and so it has a limit there.
 Properties  (H1) and (H2) are easy to verify. Properties
(H3) and \eqref{Anx} are obtained as follows.
For every $n\in \N$ we have
$$
 \begin{aligned}
& d_{C^r}    \left( (\A^n_y)^{-1} \circ \A^n_x, \, \Id \right)= d_{C^r}   
 \left( (\A^n_y)^{-1} \circ \A^n_x, \,  (\A^0_y)^{-1}\circ  \A^0_x  )\right)\le \\
&\le \sum_{k=0}^{n-1} d_{C^r}   \left( (\A_y^k)^{-1}  \circ \A_x^k ,\, (\A_y^{k+1})^{-1} \circ \A_x^{k+1}\right) 
\le  K_\e'' d_X(x,y)^{\beta\rho} \sum_{n=0}^\infty \theta^n \le  c\, d_X(x,y)^{\beta\rho}. \end{aligned}
$$
Taking the limit as $n\to\infty$  we obtain  $ d_{C^r}(H^s_{\,x,y},\Id)  \le  c\, d_X(x,y)^{\beta\rho}.$
 $\QED$



\subsection{Boundedness of the cocycle.} $\;$
We assume that the periodic data set bounded in $|\cdot |_{C^q}$ 
in Proposition \ref{Cr bdd} and bounded in $|\cdot |_{C^1}$ in Proposition \ref{C1 bdd}.

\begin{proposition} \label{Cr bdd}
Suppose that a cocycle $\A$ satisfies the assumptions of Proposition \ref{holonomies}
and its periodic data set $\A_P$ is bounded in $\dqm$, where
$q=k+\gamma$, $k\in \N$, $0<\gamma<1$.  
Then its value set $\A_X$ is bounded in 
 $\drm$ for any $r<q$.
\end{proposition}

\begin{proof} 
The base systems that we are considering satisfy the following {\it closing property.}

\begin{lemma} (Anosov Closing Lemma\,  \cite[6.4.15-17]{KH}) \label{closing}\,
Let $(X,f)$ be a topologically mixing diffeomorphism 
of a locally maximal hyperbolic set. Then there exist  
constants $c,\, \delta' >0$ such that for any $x \in X$ and $k\in\N$ with 
$\dist (x, f^k x) < \delta'$ there exists a periodic point $p \in X$ with 
$f^k p =p$ such that the orbit segments $x, fx, ... , f^k x$ and 
$p, fp, ... , f^k p\,$ satisfy
$$
d_X (f^i x, f^i p) \le c\,  d_X (x, f^k x)\;\,\text{ for every $i=0, ... , k$.}
$$
\end{lemma}
\noindent For subshifts of finite type 
this property can be observed directly.

Since the map $f$ is transitive,  we can consider a point $z\in X$ whose orbit
$O(z)=\{ f^n z:\; n\in \Z \}$ is dense in $X$. 
Let $f^{n_1}z$ and  $f^{n_2}z$ be two points of $O(z)$ with 
$\delta:=\dist(f^{n_1}z, f^{n_2}z)< \delta'$, where $\delta'$ is 
as in  Lemma \ref{closing}.
We assume that $n_1<n_2$ and denote 
$$
w=f^{n_1}z \quad\text{and} \quad k=n_2-n_1, \quad \text{so that }\;
\delta =\dist(w, f^k w)< \delta'. 
$$
Then there exists $p\in X$ with $f^k p =p$ 
such that $\dist(f^i w, f^i p) \le c\delta$ for $i=0, ... , k$.
\vskip.05cm

We fix $\e>0$ and take $r_1=q-\e$. 
Let $y$ be the point of intersection of  $W^s_{loc}(p)$ and  $W^u_{loc}(w)$.
Then by Proposition \ref{holonomies} (H3)  there exists a constant $c_1$
independent of $p$ and $y$ such that 
$$
   d_{C^{r_1}}((\A^k_p)^{-1}\circ \A^k_y,\, \Id)\le  c_1\delta^{\beta \e}.
$$
We use Lemma \ref{distance} with $g=\A^k_p$, $h_1 =(\A^k_p)^{-1}\circ \A^k_y$, 
and $h_1 =\Id=\tg.$
Since the set $\A_P$ is bounded in $\dqm$, all norms in \eqref{dist} and \eqref{delta 0}
are similarly uniformly bounded and, in particular, condition \eqref{delta 0} is satisfied
if $\delta'$ is chosen small enough. We conclude  that 
$$
   d_{C^{r_1}}(\A^k_y,\, \A^k_p) \le
   c(|\A_p^k|_{C^r})\cdot d_{C^{r_1}}((\A^k_p)^{-1}\circ \A^k_y,\, \Id)^{\e} \le  c_1\delta^{\beta \e^2}.
$$
Thus the set $\{ |\A_y^k |_{C^{r_1}}\}$ is bounded by \eqref{compare 2}.
A similar argument, using unstable holonomies, shows that the set  $\{|A_w^k |_{C^{r_2}}\}$ 
is  bounded for $r_2=r_1-\e=q-2\e.$

Therefore, there exists a constant $c_2$ such that whenever 
$\delta:=\dist(f^{n_1}z, f^{n_2}z)< \delta'$ with $k=n_2-n_1>0$,
we have $|\A^k_{f^{n_1}z}|_{C^{r_2}}<c_2$.

We take $m\in \N$ such that the set  $\{f^j z:\; |j|\le m\}$ is 
$\delta'$-dense in $X$ and set
$$
c_m =\max \,\{\,|\A^j_z|_{C^{r_2}}:\; |j|\le m\}.
$$
Let  $n>m$. Then there exists $j$ with $|j|\le m$, such that $d_X (f^n z, f^j z)\le \delta'$.\,
Using Lemma \ref{dlLO lemma}, we obtain that
$$
\begin{aligned}
\|\A^n_z\|_{C^{r_2}} & = \|\A^{n-j}_{f^j{z}} \circ \A^j_z\|_{C^{r_2}} \le
M_{r_2} (1+ \|\A^j_z\|_{C^{r_2}})^{r_2} \|\A^{n-j}_{f^j{z}}\|_{C^{r_2}} \le\\
&\le M_{r_2} (1+ c_m)^{r_2} \cdot c_2 \le c_3.
 \end{aligned}
$$
The cases of $\{\|(\A^n_z)^{-1}\|_{C^{r_2}}\}$ and  of $n<-m$ are similar.
Thus there exists a constant $c$ such that 
$$
|\A^n_z|_{C^{r_2}}\le c\quad\text{for all }n\in \Z.
$$
Since $O(z)$ is dense in $X$ and $|\A^n_z|_{C^{r_2}}$  is continuous on $X$ for each $n$,
we conclude that $|\A^n_x|_{C^{r_2}}$ is uniformly bounded in $x \in X$ and $n \in \Z$. 
Since $\e$ is arbitrary, the proposition follows.
\end{proof}

\begin{proposition} \label{C1 bdd}
Suppose that a cocycle $\A$ satisfies the assumptions of Proposition \ref{holonomies}
with $q=1+\gamma$, $\gamma>0$,
and its periodic data set $\A_P$ is bounded in $\donem$.  
Then its value set $\A_X$  is also bounded in  $\donem$.
\end{proposition}

Since $\A_X$ is bounded in $|\cdot |_{C^0}$, it suffices to show that the first derivatives
are uniformly bounded.
As in the previous proof, we consider the points $z, w, p$, and $y= W^s_{loc}(p)\cap W^u_{loc}(w)$.
Then by \eqref{Anx} there exists a constant $c$ independent of $p$ and $y$
such that 
$$
  d_{C^1}((\A^k_p)^{-1}\circ \A^k_y, \,\Id)\le c \quad\text{and hence}\quad
  \|D_t \left( (\A^k_p)^{-1}\circ \A^k_y\right)^{\pm 1}\| \le c'
$$
for some $c'$. Since $\A^k_y= \A^k_p \circ ((\A^k_p)^{-1}\circ \A^k_y)$
and the set $\A_P$ is bounded in $|\cdot |_{C^1}$, we conclude that 
$\|D_t(\A_y^k)^{\pm 1}\|$ is bounded uniformly in $t$, $k$, and $y$, as above,
 and thus the set $\{\A_y^k\}$ is bounded in $|\cdot |_{C^1}$.
Then a similar argument shows that the set  $\{\A_w^k\}$ is also bounded in $|\cdot |_{C^1}$.
Therefore, there exists a constant $c''$ such that whenever 
$\dist(f^{n_1}z, f^{n_2}z)< \delta'$ with $k=n_2-n_1>0$,
we have $|\A^k_{f^{n_1}z}|_{C^1}<c''$.

We take $m\in \N$ such that the set  $\{f^j z:\; |j|\le m\}$ is 
$\delta'$-dense in $X$ and let
$$
c_m =\max \,\{\,|\A^m_z|_{C^1}:\; |j|\le m\}.
$$
Since for any $n>m$
there exists $j$, $|j|\le m$, such that $d_X (f^n z, f^j z)\le \delta'$,\,
we have
$$
  \|D_t(\A^n_z)\| =
   \|D_{\A^{n-j}(t)}(\A^j_z) \circ D _t(\A^{n-j}_{f^jz})\|\le |\A^j_z|_{C^1} 
   \cdot |\A^{n-j}_{f^jz}|_{C^1}\le c_mc''
$$
and similarly for $\|D_t(\A^n_z)^{-1}\|$ and for  $n<-m$. 
Thus the set $\{ |\A^n_z |_{C^1}:\; n\in \Z\}$ is bounded, and 
it follows that $\A_X$ is bounded $|\cdot |_{C^1}$.
$\QED$
\vskip.2cm
This completes the proof of Theorem \ref{main 1}.


\section{Proof of Theorem  \ref{main 2}}

We recall that $q=1+\gamma$, where $0<\gamma <1$.  We fix $0<\a<\gamma$ and set  $r=1+\alpha$.
Since the periodic data set $\A_P$ is bounded in $\dqm$, 
 the cocycle $\A$ satisfies the conclusion of Proposition \ref{holonomies}, and in particular 
 has stable and unstable holonomies in $\drm$. Also, the set of values $\A_X$ is bounded 
in $\drm$ by Proposition \ref{Cr bdd}.
\vskip.1cm

As in Section \ref{slow},  we consider the linear cocycle $\B_{(x,t)}= D_t\A_x $ 
over the map $F(x,t)=(f(x), \A_x (t))$ on the vector bundle $\E$  over $X\times \M$ with fiber
$\E_{(x,t)} =T_t\M$.
\vskip.1cm

The Proof of Theorem \ref{main 2} is organized as follows. In Section \ref{measurable}
we construct  an everywhere defined  bounded measurable family $\hat \tau$ of inner products 
on the vector bundle $\E$ invariant under the cocycle $\B$. 
In Section \ref{along fibers} we show that $\hat \tau$ is H\"older continuous along each 
fiber $\M_x$ in $X\times \M$.  In Section \ref{essential inv}
we establish essential invariance of $\hat \tau$  under the  holonomies 
of the linear cocycle $\B$ along the stable and unstable sets in the skew product.
In Section \ref{Holder tau}, we consider a natural invariant measure $\hat \mu$
 for the skew product that projects to the measure of maximal entropy $\mu$ for $X$.
Then using the above essential invariance with respect to $\hat \mu$ we show that 
$\hat \tau_x$, the restriction of $\hat \tau$ to $\M_x$,  is $\mu$-essentially invariant
under the stable and unstable holonomies of the cocycle $\A$, as a Riemannian metric on 
the whole fiber $\M_x$. This yields essential H\"older continuity of $\hat \tau_x$ along
the stable and unstable leaves in $X$, and hence global H\"older continuity.

The arguments in Sections \ref{measurable} and \ref{essential inv} are similar to those
in \cite{S,KS10,KS13} for conformality of linear cocycles over hyperbolic and partially 
hyperbolic systems. However, the skew product $F$ is not a volume preserving
 smooth partially hyperbolic system as in \cite{KS13}, so we adapt arguments
 to our case. More importantly, $F$ is not accessible and so we need arguments in 
 Section \ref{along fibers} to show regularity along the ``center" $\M$ 
 and in Section \ref{Holder tau} to deduce global regularity.

\subsection {An invariant  bounded measurable family $\hat \tau$ of inner products on $\E$.} \label{measurable} $\;$ \\
In this section, it suffices to assume that the cocycle $\A$ is bounded in $|\cdot |_{C^1}$,
and hence $\|\B_{(x,t)}^n\|$ is uniformly bounded in $(x,t)\in X\times \M$ and $n\in \Z$.
\vskip.05cm

The space $\T^m$ of inner products on $\R^m$
identifies with the space of real symmetric positive definite $m\times m$ matrices, 
which is isomorphic to $GL(m,\R) /SO(m,\R)$. 
The group $GL(m,\R)$ acts transitively on $\T^m$
via $A[E] = A^T E \, A,$ where $A\in GL(m,\R)$ and $E \in \T^m.$ 
The space $\T^m$ is a Riemannian symmetric space 
of non-positive curvature when equipped with a certain $GL(m,\R)$-invariant 
metric \cite[Ch.\,XII, Theorem 1.2]{L}. 
Using the background Riemannian metric on $\E$, we can identify an inner product
with a symmetric linear operator. For each $(x,t)\in X\times \M$, 
we denote the space of inner products on $\E_{(x,t)}$ by $\T_{(x,t)}$, and so
we obtain a bundle $\T$ over $X\times \M$ with
fiber  $\T_{(x,t)}$. We equip the fibers of $\T$ with the Riemannian 
metric mentioned above.  We call a continuous (H\"older continuous, measurable) 
section of $\T$ a continuous  (H\"older continuous, measurable)
 Riemannian metric on $\E$. 
A  metric metric $\tau$ is called {\em bounded} 
if the distance between $\tau_{(x,t)}$ and $\tilde \tau_{(x,t)}$ is uniformly 
bounded on $X\times \M$ for a continuous metric $\tilde \tau$ 
on $\E$. 
For the linear cocycle $\B$, the pullback of an inner product
$\tau_{F(x,t)}$ on $\E_{F(x,t)}$ to $\E_{(x,t)}$ is given by
$$
\left( \B_{(x,t)}^* (\tau_{F(x,t)}) \right) (v_1, v_2)= 
\tau_{F(x,t)} \left (\B_z (v_1),\,(\B_z)v_2 \right) \quad\text{for } v_1, v_2 \in \E_{(x,t)}.
$$
We say that a metric  $\tau$ is $\B${\em -invariant}\,
if $\B^\ast(\tau) = \tau$.

Let $\tau$ be a continuous metric on $\E$.
We consider the set 
$$
S(x,t)=\{\,(\B^n_{(x,t)})^*(\tau_{F^n (x,t)}): \;n\in\Z\,\}  \subset \T_{(x,t)}.
$$ 
Since the cocycle $\B$ is bounded, the sets $S(x,t)$ have 
uniformly bounded diameters.  Since the space $\T_{(x,t)}$ has 
non-positive curvature,  for every $(x,t)$ there exists a 
unique smallest closed ball containing $S(x,t)$ \cite[Ch.\,XI, Theorem 3.1]{L}.
We denote its center by $\hat \tau_{(x,t)}$. By the construction,
 the metric $\hat \tau$ is invariant under $\B$.

For any $k\ge 0$, the set 
\begin{equation}\label{Sk}
S^k(x,t)=\{\,(\B^n_{(x,t)})^*(\tau_{F^n (x,t)}): \;|n| \le k \}
\end{equation}
depends continuously on $(x,t)$ in Hausdorff distance, and so does 
the center $\tau^k_{(x,t)}$ of the smallest ball containing $S^k(x,t)$
by Lemma \ref{centers} below.
Since $S^k(x,t) \to S(x,t)$  in Hausdorff distance as $k \to \infty$ for any $(x,t)$,
the metric $\hat \tau$ is the pointwise limit of continuous 
metrics $\tau^k_{(x,t)}$, 
and hence $\hat \tau$ is Borel measurable.

\begin{lemma}\label{centers}
Let $S_1$ and $S_2$ be bounded sets in $\T^m$ and let  $B(c_1, r_1)$ and $B(c_2, r_2)$ 
be the smallest 
closed balls  containing $S_1$ and $S_2$, respectively. 
Then $\,d_{\T^m}(c_1,c_2)  \le \,d_\text{H} (S_1,S_2)$, where $d_H$ is the Hausdorff distance.

\end{lemma}

\begin{proof}
Suppose $d_\text{H} (S_1,S_2) < \e$. Then,
since $S_1\subset B(c_1, r_1)$, we have
$S_2\subset B(c_1,r_1+\e)$,
and hence $B(c_2, r_2) \subset B(c_1,r_1+\e)$ by minimality. 
Similarly, $B(c_1, r_1) \subset B(c_2,r_2+\e)$.

If $d_{\T^m} (c_1,c_2) > \e$, we obtain a contradiction. Indeed, suppose $r_2\ge r_1$. 
Any two points in $\T^m$ lie on a unique geodesic, which is isometric to $\R$.
Therefore there exists a point $c$ on the geodesic 
through $c_1$ and $c_2$ with 
$$
d_{\T^m}(c,c_2)=r_2 \;\text{ and }\; d_{\T^m}(c,c_1)=d_{\T^m}(c,c_2)+d_{\T^m}(c_2, c_1) >r_2+\e\ge r_1+\e.
$$
This point is in $B(c_2,r_2)$, but not in $B(c_1,r_1+\e)$. Thus $\dist (c_1,c_2) \le \e$.
\end{proof}


\subsection{ H\"older continuity of $\hat \tau$ along the fibers $\M_x$ in $X\times \M$.} \label{along fibers}
$\;$ We recall that by Proposition \ref{Cr bdd} the value set $\A_X$ is bounded in $\drm$, where $r=1+\a$.

We start with an $\alpha$-H\"older continuous metric $\tau$ on $\V$ and obtain 
a bounded Borel measurable metric $\hat \tau$, as in the previous section.
Below we show that $\hat \tau$ is H\"older along the fibers.
We denote by $\hat \tau_x$ the restriction of the metric $\hat \tau$ to the
fiber $\M_x$.

\begin{proposition}  
For each $x \in X$ the metric $\hat \tau_x$ is $\a$-H\"older continuous on $\M_x$, 
more precisely, there exists a constant $c$ such that 
\begin{equation}\label{tau H}
d_{\T} (\hat \tau(x, t), \hat \tau(x, t'))\le c\, d_\M (t, t')^\alpha 
\quad\text{for all } x\in X \text{ and } t, t' \in \M.
\end{equation}

\end{proposition}

\begin{proof}
We use the following lemma. It was proven in \cite{KS10} for conformal structures
rather than inner products, but the proof works without significant modifications.

\begin{lemma}\label{A*} (cf. \cite[Lemma 4.5]{KS10})
Let $\tau$ be an inner product on $\R^m$ and $B$ be a linear
transformation of $\,\R^m$ sufficiently close to
the identity. Then 
  $$
     d_\T(\tau, B(\tau)) \le c_1(\tau)\cdot\|B-\Id\,\|,
  $$
where the function $c_1(\tau)$ is bounded on compact sets in $\T^m$. 
\end{lemma}

In the chain of inequalities below, we use the following: 
the pullback action is an isometry; 
the metric $\tau$ is $\a$-H\"older continuous and in particular bounded; 
Lemma \ref{A*}; 
and the fact that since $\A_X$ is bounded in $|\cdot|_{C^{1+\a}}$,
the norm $\|(\B^n_{(x,t)})^{\pm 1}\|$ is bounded and there is a constant $c_2$ such that
$$
\|  \B^n_{(x,t)} - \B^n_{(x,t')} \| \le c_2\, d_\M (t, t')^\alpha.
$$
For each $n\in \Z$ we have 
$$
\begin{aligned}
& d_{\T}\left((\B^n_{(x,t)})^*(\tau_{F^n (x,t)}), (\B^n_{(x,t')})^*(\tau_{F^n (x,t')})\right) \le \\
 & \le d_{\T} 
 \left( (\B^n_{(x,t)})^*(\tau_{F^n (x,t)}), (\B^n_{(x,t)})^*(\tau_{F^n (x,t')})\right) +\\
&\hskip5cm + d_{\T} \left( (\B^n_{(x,t)})^*(\tau_{F^n (x,t')}), (\B^n_{(x,t')})^*(\tau_{F^n (x,t')})\right) \le\\
 &\le d_{\T} \left(\tau_{F^n (x,t)}, \tau_{F^n (x,t')} \right) + 
 d_{\T} \left( \tau_{F^n (x,t')}, (\B^n_{(x,t)})^{-1} \B^n_{(x,t')})^*(\tau_{F^n (x,t')})\right) \\
&\le c_3 \,d_\M (t, t')^\alpha  +
 c_1 (\tau_{F^n (x,t')}) \| (\B^n_{(x,t)})^{-1} \B^n_{(x,t')} -\Id\|  \le \\
 &\le c_3\, d_\M (t, t')^\alpha  + c_4\| (\B^n_{(x,t)})^{-1}  \| \cdot \|  \B^n_{(x,t)} - \B^n_{(x,t')} \| \le \\
 & \le c_3\, d_\M (t, t')^\alpha  + c_4c_5 c_2 \, d_\M (t, t')^\alpha = c_6 \, d_\M (t, t')^\alpha.
 \end{aligned} 
$$
Thus for each $n$,
$$
  d_{\T} \left( (\B^n_{(x,t)})^*(\tau_{F^n (x,t)}), (\B^n_{(x,t')})^*(\tau_{F^n (x,t')})\right)
  \le c\, d_\M (t, t')^\alpha \quad\text{for all }t,t'\in \M,
$$
where the constant $c=c_6$ is independent of $x$ and $n$.
It follows immediately that the Hausdorff distance between the sets 
$S^k(x, t)$ and  $S^k(x, t')$ given by \eqref{Sk} satisfies
$$
d_\text{H} (S^k(x, t), S^k(x, t'))\le c\, d_\M (t, t')^\alpha 
 \quad\text{for all } x\in X \text{ and } t, t' \in \M.
$$
Hence by Lemma \ref{centers} the center $\tau^k_{(x,t)}$ of the smallest closed ball 
containing $S^k(x,t)$ also satisfies
$$
d_{\T}(\tau^k_{(x,t)}, \tau^k_{(x,t')} )\le  c\, d_\M (t, t')^\alpha \quad\text{for all }
x,k, t, t'.
$$
Passing to the limit as $k\to\infty$, we obtain \eqref{tau H}.
\end{proof}


\subsection{Stable and unstable sets for $F$} 
We consider the map $F(x,t)=(f(x), \A_x (t))$ of the set $X\times \M$.
While it is not partially hyperbolic 
in the classical sense, we can define the stable sets $\tilde W^s$ for $F$ 
using the  stable holonomies $H^s_{x,y}$ given by Proposition \ref{holonomies}.
The unstable sets $\tilde W^u$ are defined similarly.
For any $(x,t)\in X\times \M$, we set
$$
  \tilde W^s(x,t)=\{(y,t')\in X\times \M: \; y\in W^s(x), \;\, t'= H^s_{x,y}(t)\}.
$$
We will only use the following contraction property for these sets:
$$
d_{X\times \M} (F^n(x,t), F^n(y,t')) \to 0 \quad\text{as }n\to \infty 
$$
for any $(x,t)\in X\times \M$ and  $(y,t')\in \tilde W^s(x,t)$. It holds since
$$
\begin{aligned}
& d_{X\times \M} (F^n(x,t), F^n(y,t')) =d_{X\times \M}  \left((f^n x, \A^n_x(t)), (f^n y, \A^n_y(t'))\right)=\\
&=d_X (f^nx, f^ny)+d_\M (\A^n_x(t),\A^n_y(t')) \to 0 \quad\text{as }n\to \infty
\end{aligned}
$$
as $d_X (f^nx, f^ny) \to 0$ and
$$
 \A^n_y(t')= H^s_{f^nx ,\,f^ny} \circ \A^n_x \circ H^s_{y,x}(t')= H^s_{f^nx ,\,f^ny} \circ \A^n_x (t),
$$
where $d_{C^r}(H^s_{f^nx ,\,f^ny},\Id) \le c\, d_X(f^nx ,f^ny)^{\beta\rho} \to 0\,$ 
by (H3) in Proposition \ref{holonomies}.
\vskip.1cm

\subsection{Essential invariance  of $\hat \tau$ under the holonomies.} \label{essential inv}
For convenience, in the remaining two sections we will use the push forward of an inner product 
by a linear map, which is defined as the pull-back by its inverse: $L_* =(L^{-1})^*$.

\vskip.1cm

First we show that $\hat \tau$ is essentially invariant under the derivatives of $H^s$ 
along the stable sets of $F$ in $X\times \M$. These derivatives can be interpreted as
stable holonomies of the cocycle $\B$. The statement and proof for the 
unstable holonomies are similar.

\begin{proposition} \label{tau H invariant} 
Let $\nu$ be an ergodic $F$-invariant probability measure on $X\times \M$.
If $\tau$ is a $\nu$-measurable  $\B$-invariant metric on $\V$, then  
$\tau$ is essentially $H^s$-invariant, more precisely, there exists an  $F$-invariant  
set $E\subset X\times \M$  with $\nu(E)=1$ such that 
 $$
\hat \tau(y,t') = (D_tH^s_{x,y})_* (\hat \tau(x,t))  \quad \text{for all }(x,t), (y,t') \in E 
\;\text{ with }(y,t') \in \tw^s_{loc}(x,t).
$$ 
\end{proposition}
\begin{proof}

To simplify the notations, we write $\tau$ for $\hat \tau$ and $d$ for $d_\T$, and
for  $y\in W_{loc}^s(x)$, $t\in \M_x$, and $t'=H^s_{x,y}(t)\in \M_y$, we set
$$
z=(x,t), \;\; z'=(y,t') \;\text{ so } z'\in \tilde W_{loc}(z),\;\; z_n=F^n (z),  \; \text{ and }\; z'_n=F^n(z').
$$
\vskip.2cm

Since  $\tau$ is $\nu$-measurable, by Lusin's Theorem 
 there exists a compact set $S\subset X\times \M$ 
with $\nu(S)>1/2$ so that $\tau$ is uniformly continuous and 
hence bounded on $S$.
Let $E$ be the set of points in $X\times \M$ for which the asymptotic frequency 
of  visiting $S$ equals $\nu(S)>1/2$. By Birkhoff Ergodic Theorem, $\nu(E)=1$.
\vskip.05cm

Suppose that  both $z$ and $z'$ are in $E$. We will show that 
$$
d\left( \tau(z'),\, (D_tH^s_{xy})_* (\tau(z)) \right) = 0, 
\;\text{ that is, }\; \tau(z')= (D_tH^s_{xy})_* (\tau(z)).
$$
Property (H2) of the holonomies, 
$H^s_{x,y}= (\A^n_y)^{-1}\circ H^s_{f^nx ,\,f^ny} \circ \A^n_x,\;$ 
implies
$$
\begin{aligned}
  & D_t H^s_{x,y}  =
  D_{ \A^n_y (t')}(\A^n_y)^{-1}\circ D_{\A^n_x (t)}\,H^s_{f^nx ,\,f^ny} \circ D_t\A^n_x \,= \\
  & = (\B^n_{(y,t')})^{-1} \circ D_{\A^n_x (t)}\,H^s_{f^nx ,\,f^ny} \circ \B^n_{(x,t)}
  = (\B^n_{z'})^{-1} \circ D_{\A^n_x (t)}\,H^s_{f^nx ,\,f^ny} \circ \B^n_{z}.
  \end{aligned}
$$
Since the metric $\tau$ is invariant under the cocycle $\B$,
and $\B$ induces an isometry on the space of inner products,  we have
\begin{equation}\label{long}
  \begin{aligned}
&d\left( \tau(z'),\, (D_tH^s_{xy})_* (\tau(z)) \right) = \\
&=d \left( \tau(z'),\, 
((\B^n_{z'})^{-1})_* \,(D_{\A^n_x (t)}\,H^s_{f^nx ,\,f^ny} \circ \B^n_{z})_*(\tau(z)) \right) =\\
&=d \left( (\B^n_{z'})_* (\tau(z')),\,  (D_{\A^n_x (t)}\,H^s_{f^nx ,\,f^ny})_*  
(\B^n_{z})_*(\tau(z)) \right) = \\
& = d \left( \tau(z_n'),  (D_{\A^n_x (t)}\,H^s_{f^nx ,\,f^ny})_* (\tau(z_n)) \right)=\\
&\le  d\,( \tau(z_n'), \tau(z_n)) + 
d \left( \tau(z_n), (D_{\A^n_x (t)}\,H^s_{f^nx ,\,f^ny})_* (\tau(z_n)) \right) . 
 \end{aligned}
\end{equation}
Since $z,z'\in E$, there exists a sequence $\{ n_i \}$ such that 
 both $z_{n_i}$ and $z'_{n_i}$ are in $S$ for each $i$.
Since $z'\in \tw^s_{loc}(z)$, $\,d_{X\times \M}(z_{n_i}, z'_{n_i})\to 0$ and hence 
$d(\tau(z_{n_i})), \tau(z'_{n_i}))\to 0$ by uniform continuity of $\tau$ on $S$. 
\vskip.05cm 

By property (H3) of holonomies, $d_{C^r}(H^s_{\,x,y},\Id) \leq c\,d_X (x,y)^{\beta \rho},$
 where $c$ is independent of $x$   and $y\in W^{s}_{\text{loc}}(x).$
 Hence 
 $$
   \| D_{\A^n_x (t)}\,H^s_{f^nx ,\,f^ny} -\Id \| \le \kappa\, d_{C^r}(H^s_{f^nx ,\,f^ny},\Id) \le
   \kappa  c\,d_X (x_n,y_n)^{\beta \rho} \to 0 \;\text{ as }n\to \infty.
 $$
Using Lemma \ref{A*} and boundedness of $\tau$ on $S$, we conclude that 
the last term in \eqref{long} tends to 0 as $n\to \infty$. Therefore,
$
d \left( \tau(z'),\, (D_tH^s_{xy})_* (\tau(z)) \right) = 0.
$
\end{proof}


\subsection{H\"older continuity of $\hat \tau$} \label{Holder tau}
Now we consider a particular measure on $X\times \M$.
Let $\mu $ be the measure of maximal entropy for $(X,f)$.
For each $x\in X$, we have the normalized volume $m_x$ on the fiber $\M_x$ 
 induced by the metric $\hat \tau$, which is H\"older continuous on $\M_x$.
We consider the measure $\hat \mu$ on $X\times \M$ given by
$\hat\mu = \int m_x \,d\mu(x)$, that is, for any Borel measurable set $S\subset X\times \M$,
$$
\hat\mu (S) = \int_X m_x (S\cap \M_x)\, d\mu(x).
$$
\vskip.05cm

Clearly, the measure $\hat \mu$ is $F$-invariant. While it is not necessarily ergodic, 
we can consider its ergodic components
and the corresponding partition $\xi$ of $X\times \M$. The Hopf argument yields that, up to a set 
of measure zero, every local stable set is contained in an element of $\xi$. 
So we can apply Proposition \ref{tau H invariant} to ergodic components of $\hat \mu$ 
we obtain the following.

\begin{corollary}  \label{z'}
There exists a set $\hat G\subset X \times \M$ with 
$\hat \mu(\hat G)=1$ such that 
$\hat \tau$ on $\hat G$ is invariant under the holonomies, that is, 
$$
\hat \tau(z') = (D_tH^s_{x,y})_* (\hat \tau(z))\, 
\text{ for all }z=(x,t)\in \hat G\, \text{ and all } z'=(y,t') \in \hat G \cap \tw^s_{loc}(z). 
$$
\end{corollary}

Now we establish $\mu$-essential  invariance of $\hat \tau_x$, as a Riemannian metric on 
the whole fiber $\M_x$, under the stable and unstable holonomies of the cocycle $\A$ over $X$.

\begin{proposition}  \label{suInv}
There exists a set $G\subset X$ with $\mu(G)=1$ such that 
or any $x,y,y' \in G$  with $y \in W^s_{loc}(x)$ and  $y' \in W^s_{loc}(x)$,
the diffeomorphisms
$$
H^s_{x,y}: (\M_x, \hat \tau_x) \to (\M_y ,\hat \tau_y) \;\text{ and }\;
H^u_{x,y'}: (\M_x, \hat \tau_x) \to (\M_{y'} ,\hat \tau_{y'}) \;\text{ are isometries.}
$$
\end{proposition}

\begin{proof} 
We will obtain a set $G^s$ of full measure for the stable holonomies. A similar argument gives
a full measure set $G^u$ for the unstable holonomies, and $G=G^s \cap G^u$.
\vskip.1cm

Let $\hat G\subset X\times\M$ be as in Corollary \ref{z'} and let
$G^s=\pi(\hat G)\subset X$, where $\pi$ is the projection from
$X\times \M$ to $X$.  Then we have $\mu(G^s)=1$ and 
$m_x(\M_x \cap \hat  G)=1$ for $\mu$ almost all $x\in G^s$.
Discarding a set of measure zero, we can assume that $\hat  G$ is $F$-invariant
and 
$$
m_x(\M_x \cap \hat  G)=1 \quad\text{for all }\, x\in G^s.
$$ 

Now we show that if $x,y\in G^s$ and $y\in W^s_{loc}(x)$,
then $H^s_{x,y}$ is an isometry between the fibers $(\M_x, \hat \tau_x)$ and $(\M_y,\hat \tau_y).$
Since $H^s_{x,y}: \M_x \to \M_y$ is a diffeomorphism, it maps  sets of zero 
volume to sets of zero volume. The set 
$$
E= (\M_x \cap \hat G)\cap \left( (H^s_{x,y})^{-1} (\M_y \cap \hat G)\right)
$$
satisfies  $E\subset \hat G$, $ H^s_{x,y}(E)\subset \hat G$, and $m_x(E)=1$,
and in particular $E$ is dense in  $\M_x$.
By Corollary \ref{z'} we have 
$$
\hat \tau(z')=(D_tH^s_{x,y})_*(\hat \tau(z)) \quad\text{for all }z\in E.
$$
Thus $D H^s_{x,y}$ is isometric on the dense set $E$ and, as the Riemannian metrics
$\hat \tau_x$ and $\hat \tau_y$ are $\a$-H\"older continuous along the fibers, we conclude 
that the diffeomorphism 
$$
H^s_{x,y}: (\M_x, \hat \tau_x) \to (\M_y ,\hat \tau_y) \;\text{ is an isometry for all }x,y \in G^s 
\;\text{ with }y \in W^s_{loc}(x).
$$ 
\end{proof}

We denote by $\T(\M,\a)$ the space of $\a$-H\"older continuous Riemannian metrics on $\M$ 
equipped with $C^\a$ distance $d_{\T \alpha}$.
Then the $\mu$-essential invariance of $\hat\tau$ yields $\mu$-essential $\beta\rho$-H\"older continuity of
$\hat\tau$ as a function from $X$ to $\T(\M,\a)$ along the stable and unstable leaves in $X$. 

\begin{corollary}  \label{Holder}
The function $x \mapsto \hat\tau_x$ is $\beta\rho$-H\"older continuous on $G$ 
along the stable and unstable leaves in $X$ as a function from $X$ to $\T(\M,\a)$, that is 
\begin{equation}\label{Hold hol tau}
d_{\T \alpha}(\hat \tau_x,\hat \tau_y) \le  C\, d_X(x,y)^{\beta\rho} \quad \text{for all }x,y \in G\, \text{ with }y \in W^{s/u}_{loc}(x).
\end{equation}

\end{corollary} 

\begin{proof}
By  Proposition \ref{holonomies} (H3) the holonomies $H^{s/u}$ are $\beta\rho$-H\"older 
continuous in $\drm$, $r=1+\alpha$. By Proposition \ref{suInv}, $H^{s/u}_{x,y} (\hat \tau_x) 
=\hat \tau_y$  for all $x,y \in G$
 with $y \in W^{s/u}_{loc}(x)$. Now 
using Lemma \ref{A*} and  boundedness of $\hat \tau$,  we obtain 
$$
d_\T(\hat \tau_x,\hat \tau_y) \le c_1 \,d_{C^r}(H^{s/u}_{x,y}, \Id)\le c_1 c\, d_X(x,y)^{\beta\rho}.
$$
\end{proof}

Now the local product structure argument shows that $\hat \tau$ coincides 
$\mu$ almost everywhere with a $\beta\rho$-H\"older continuous stable and unstable 
holonomy invariant  function $\tau :X \to \T(\M,\a)$.
We consider a small open set $U$ in $X$ 
with the product structure of stable and unstable leaves, that is 
$$
U=W^s_{loc}( x_0)\times W^u_{loc}( x_0)\overset{def}{=}\,
\{W^s_{loc}( x)\cap W^u_{loc}( y)\; | \; x\in W^s_{loc}( x_0), \;y\in W^u_{loc}( x_0)\}.
$$
We recall that the measure of maximal entropy $\mu$  is equivalent to the product 
of its conditional measures on $W^s_{loc}( x_0)$
and $W^u_{loc}( x_0)$, which have full support on the corresponding leaves.
Therefore for $\mu$ almost all local stable leaves in $U$, the set of  points of 
$G$ on the leaf has full conditional measure, and hence  full support. 
Without loss of generality, we can assume that $G$ has no points on the other leaves.
Hence for any two  points $x$ and $y$ in $G\cap U$ there exists a point 
$w\in W^s_{loc}(x)\cap G$  such that 
$W^u_{loc}(w)\cap W^s_{loc}(y)$ is also in $G\cap U$.
Then \eqref{Hold hol tau} and the local product structure of the stable 
and unstable manifolds yield that for all $x,y \in G\cap U$ we have
$$
d_\T(\hat \tau(x), \hat \tau (y))\le c_3 \,d_X(x,y)^{\beta\rho}.
$$
Since this estimate holds for all $x,y \in G$, which is dense in $X$,
$\hat \tau$ extends to a $\beta\rho$-H\"older continuous function $\tau :X \to \T(\M,\a)$,
which is also invariant under the holonomies and the cocycle.
As a function on $X\times \M$, $\,\tau$ is $\gamma$-H\"older continuous with
 $\gamma = \min \,\{\a, \beta\rho\}.\,$
 This completes the proof of Theorem \ref{main 2}.



\vskip.7cm
\begin{thebibliography}{99}

\bibitem[AKL18]{AKL} A. Avila, A. Kocsard, and X. Liu.
{\em Liv\v{s}ic  theorem for diffeomorphism cocycles}.
To appear in GAFA.

\bibitem[BK15]{BK} L. Backes and A. Kocsard. 
{\em Cohomology of dominated diffeomorphism-valued cocycles over hyperbolic systems.}  
Ergodic Theory Dynam. Systems, 36 (2015) 1703-1722.

\bibitem[dlLO98]{dlLO}  R. de la Llave and R. Obaya.
{\em Regularity of the composition operator in spaces of H\"older functions.}
Discrete and Continuous Dynamical Systems. 5 (1999), no. 1, 157-184.

\bibitem[dlLW10]{LW}  R. de la Llave and A. Windsor. 
              {\em Liv\v{s}ic theorem for non-commutative groups including groups  
              of diffeomorphisms, and invariant geometric structures.} 
               Ergodic  Theory   Dynam. Systems, 30, no. 4  (2010),  1055-1100.               

\bibitem[Gu]{Gu} M. Guysinsky. {\em Liv\v{s}ic theorem for cocycles with values in the group of diffeomorphisms}. Preprint.

\bibitem[H18]{H} S. Hurtado.  {\em Examples of diffeomorphism group cocycles with no periodic approximation. } \\To appear in Proceedings of the AMS. arXiv:1705.06361

\bibitem[Ka11]{K11} B. Kalinin. {\em Liv\v{s}ic theorem for matrix cocycles.} 
Annals of Mathematics, 173 (2011), no. 2, 1025-1042.

\bibitem[KP16]{KP} A. Kocsard and R. Potrie.  {\em Liv\v{s}ic theorem for low-dimensional diffeomorphism cocycles}, Comment. Math. Helv. 91 (2016), 39-64.

\bibitem[KaS10]{KS10}  B. Kalinin and V. Sadovskaya.  
{\em Linear cocycles over hyperbolic systems and criteria of conformality.} 
Journal of Modern Dynamics, vol. 4 (2010), no. 3, 419-441. 

\bibitem[KS13]{KS13} B. Kalinin and V. Sadovskaya. 
{\em Cocycles with one exponent over partially hyperbolic systems.}
Geometriae Dedicata, Vol. 167, Issue 1 (2013), 167-188. 


\bibitem[KtH]{KH}  A. Katok and B. Hasselblatt. 
              {\em Introduction to the modern theory of dynamical systems.}  \\
              Encyclopedia of Mathematics and  its Applications
              {\bf 54}, Cambridge University Press,  1995.
 
\bibitem[KtN]{KtN}  A. Katok and V. Nitica.  {\em Rigidity in Higher Rank Abelian Group Actions: Volume 1, Introduction and Cocycle Problem.}
Cambridge University Press,  2011. 

\bibitem[La]{L} S. Lang. {\em Fundamentals of Differential Geometry.} New York: Springer-Verlag, 1999.
         
 \bibitem[Liv71]{Liv1} A. N. Liv\v{s}ic. 
            {\em Homology properties of Y-systems.}  Math. Zametki 10, 758-763, 1971.

\bibitem[Liv72]{Liv2} A. N. Liv\v{s}ic.  {\em Cohomology of dynamical systems.} 
                      Math. USSR Izvestija 6, 1278-1301, 1972. 

\bibitem[NT95]{NT95} V. Nitica and A. T\"or\"ok. {\em Cohomology of dynamical systems 
                       and rigidity of partially hyperbolic actions of higher-rank lattices.} 
                       Duke Math. J., 79(3), 1995, 751-810. 

 \bibitem[NT96]{NT96} V. Nitica and A. T\"or\"ok.               
  {\em Regularity  results for the solutions of the Livsic cohomology equation with values 
  in diffeomorphism groups}.\, Ergodic Theory Dynam. Systems 16 (1996), no. 2, 325-333.

\bibitem[NT98]{NT98} V. Nitica and A. T\"or\"ok.                {\em Regularity of the transfer map for cohomologous cocycles.}                 Ergodic Theory Dynam. Systems, 18(5),  1998, 1187-1209. 

\bibitem[S05]{S}   V. Sadovskaya. On uniformly quasiconformal Anosov systems.
                      Math. Res. Lett., vol. 12 (2005), no. 3,  425-441.

\bibitem[S15]{S15} V. Sadovskaya. {\em Cohomology of  fiber bunched cocycles over hyperbolic systems.}
Ergodic Theory  Dynam. Systems, Vol. 35, Issue 8 (2015), 2669-2688.     

\bibitem[PW01]{PW} M. Pollicott and C. P. Walkden.  
                {\em Liv\v{s}ic theorems for connected Lie groups.} 
                Trans. Amer. Math. Soc., 353(7),  2001, 2879-2895. 
                  
\bibitem[Sch99]{Sch} K. Schmidt.
            {\em Remarks on Liv\v{s}icÕ theory for non-Abelian cocycles.}
             Ergodic Theory Dynam. Systems, 19(3),  1999, 703-721.

\bibitem[Schr98]{Schr} S. J. Schreiber. 
            {\em On growth rates of subadditive functions for semi-flows.}
              J. Differential Equations, 148 (1998), 334-350.

\bibitem[T06]{T} M. Taylor. {\em Existence and Regularity of Isometries}.\,
Trans. Amer. Math. Soc.,
Vol. 358, No. 6 (2006),  2415-2423.
\end{thebibliography}
\end{document}